\documentclass[10pt]{amsart}

\usepackage[english]{my-shortcuts}

\usepackage{a4wide}
\usepackage{amsmath,amsfonts,amssymb,latexsym,amsthm}
\usepackage{comment} 
\usepackage{graphicx}
\usepackage{caption} 

\newtheorem{definition}{Definition}
\newtheorem{lemma}{Lemma}

\newtheorem*{remark}{Remark}
\newtheorem{theorem}{Theorem}

\newcommand{\EE}{\mathbb{E}}

\newcommand{\PP}{\mathbb{P}}
\newcommand{\RR}{\mathbb{R}}

\newcommand{\HH}{\mathcal{H}}

\newcommand{\vari}{{\rm Var}}
\newcommand{\ve}{\varepsilon}

\renewcommand{\span}{{\rm span}}

\begin{document}

\title[On probabilistic Nyman-Beurling criteria]{On probabilistic generalizations\\ of the Nyman-Beurling criterion for the zeta function} 
\author{S\'ebastien Darses --- Erwan Hillion} 

\dedicatory{\flushright To Luis B\'aez-Duarte,\\\flushright In Memoriam.}

\address{Aix-Marseille Universit\'e, CNRS, Centrale Marseille, I2M, Marseille, France} 
\email{sebastien.darses@univ-amu.fr}

\address{Aix-Marseille Universit\'e, CNRS, Centrale Marseille, I2M, Marseille, France}
\email{erwan.hillion@univ-amu.fr}

\begin{abstract}

The Nyman-Beurling criterion is an approximation problem in the space of square integrable functions on $(0,\infty)$, which is equivalent to the Riemann hypothesis. This involves dilations of the fractional part function by factors $\theta_k\in(0,1)$, $k\ge1$. We develop probabilistic extensions of the Nyman-Beurling criterion by considering these $\theta_k$ as random: this yields new structures and criteria, one of them having a significant overlap with the general strong B\'aez-Duarte criterion. 
The main goal of the present paper is the study of the interplay between these probabilistic Nyman-Beurling criteria and the Riemann hypothesis. We are able to obtain equivalences in two main classes of examples: dilated structures as exponential $\cal E(k)$ distributions, and random variables $Z_{k,n}$, $1\le k\le n$, concentrated around $1/k$ as $n$ is growing.  
By means of our probabilistic point of view, we bring an answer to a question raised by B\'aez-Duarte in 2005: the price to pay to consider non compactly supported kernels is a controlled condition on the coefficients of the involved approximations.
\end{abstract}

\maketitle


\section{Introduction}

Open problem since Riemann's memoir in 1859, the Riemann hypothesis (RH) enjoys numerous equivalent reformulations from many areas of mathematics. We refer to two expository papers \cite{Con03} and \cite{Bal10} for discussions about various approaches. One of these stems from functional analysis, which goes back to the works of Nyman \cite{Nym50} and Beurling \cite{Beu55}, strengthened by B\'aez-Duarte \cite{BD03}. 

The Nyman-Beurling criterion is an approximation problem in the space of square integrable functions on $(0,\infty)$, which involves dilations of the fractional part function by factors $\theta_k\in(0,1)$, $k\ge1$. We develop in the current paper a new approach based on considering these dilation factors as random and possibly in the whole range $(0,\infty)$. This probabilistic point of view provides new structures and yields an answer to a question raised by B\'aez-Duarte in \cite{BD05}: It is possible to obtain a sufficient condition (implying RH) while considering analytic kernels in the general strong B\'aez-Duarte criterion introduced in \cite{BD05}.

In this introduction, we first start with basic notations. Second, we recall the known deterministic criteria. We then introduce what we call the probabilistic and the general Nyman-Beurling criteria. We finally describe the main results of our paper.



\subsection{Basic notations} \label{subsec:notations}
We adopt the following conventions and notations for

{\it Functions}
\begin{itemize}
    \item The indicator function of a set $A$ is defined as $\1_A(x) = 1$ if $x \in A$ and $\1_A(x)=0$ if not. In particular, we set $\chi=\1_{(0,1]}$. The fractional part (resp. the integral part) of a real number $x\ge 0$ reads $\left\{x\right\}$ (resp. $\lfloor x \rfloor$), and then $\left\{x\right\} = x-\lfloor x \rfloor$. For $\theta>0$, we set
\begin{equation*}
\rho_\theta(t) = \left\{\frac{\theta}{t}\right\},\quad t>0.
\end{equation*}
    \item The Riemann zeta function $\zeta$ is defined for $\sg>1$ as \bean 
    \zeta(s)=\sum_{k\ge1}\frac{1}{k^s}, \quad s=\sg+i\tau.
    \eean 
    \item The M\"obius function $\mu:\N\to \{-1,0,1\}$ is defined as $\mu(1)=1$, $\mu(n)=(-1)^r$ if $n=p_1\ldots p_r$ where $p_1,\ldots,p_r$ are distinct primes, and $\mu(n)=0$ if not (i.e. $\mu(n)=0$ if $p^2|n$).
    \item We use either Landau's notation $f=O(g)$ or Vinogradov's $f\ll_\alpha g$ to mean that $|f|\le C|g|$ for some constant $C>0$ that may depend on a parameter $\alpha$.
\end{itemize}


{\it Hilbert spaces}
\begin{itemize}
    \item The Hilbert space $H=L^2(0,\infty)$ of real valued square integrable functions for the Lebesgue measure is endowed with its scalar product (and associated norm $\|f\|_H$):
    $\d \langle f,g \rangle_H = \d \int_0^\infty f(t)g(t) dt$.
    \item Let $(f_\alpha)_{\alpha\in A}$ be a family in a Hilbert space $F$. We define $\span_F\{f_\alpha, \alpha\in A\}$
as the closure in $F$ of the vector space spanned by $(f_\alpha)_{\alpha\in A}$.
\end{itemize}


{\it Probability}
\begin{itemize}
    \item $(\Omega,\cal F,\bP)$ is a probability space. We set $\cal H = L^2(\Omega,H)$.
    \item The space of non negative random variables (r.v.) having $p$-moment is denoted by $L^p_+(\Omega)$, $p\ge1$. The expectation (resp. the variance) of $X\in L^2_+(\Omega)$ reads $\EE[X]$, or simply $\EE X$ (resp. $\vari(X)$). We also set $\|X\|_p=(\pE X^p)^{1/p}$ when $X\in L^p_+(\Omega)$.
    \item We write $X\sim \Gamma(\beta,\lb)$ to mean that the r.v. $X$ is Gamma distributed with parameters $(\beta,\lb)$.
    In that case, $\d \pE X = \frac{\beta}{\lb}$ and  $\d \vari(X)=\frac{\beta}{\lb^2}$.
    The particular case of the exponential law $\cal E(\lb)=\Gamma(1,\lb)$ of parameter $\lb$ will be one basic example throughout the paper. Recall that if $X \sim \cal E(1)$ and $\lb > 0$, then $X/\lb \sim \cal E(\lb)$.
\end{itemize}

\subsection{The deterministic criteria} 

Let us recall the fundamental identity (see e.g. \cite[(2.1.5)]{Tit86})
\bean
\int_0^\infty \left\{\frac{1}{t}\right\} t^{s-1} dt = -\frac{\zeta(s)}{s},\quad 0<\sg<1,
\eean
which gives, by means of a change of variable,
the following relationship between $\zeta$ and the Mellin transform of $\rho_\theta$:
\begin{equation}\label{eq:MellinFractional}
\widehat{\rho_\theta}(s)=\int_0^\infty \rho_\theta(t) t^{s-1} dt = - \theta^{s} \frac{\zeta(s)}{s}, \quad 0<\sg<1.
\end{equation}

RH states that the non-trivial zeros of $\zeta$ belong to the critical line $\sg=\frac{1}{2}$. See \cite{Tit86} and \cite{Ten95} for basic and advanced theory on $\zeta$.  Equation~\eqref{eq:MellinFractional} allows for different equivalent restatements of RH, which has been first done in \cite{Nym50}, \cite{Beu55}. The Nyman-Beurling criterion (NB) can be stated as follows:
\begin{theorem}[\cite{BDBLS00}] \label{th:NBcriterion}
RH holds if and only if
\begin{equation}\label{eq:NBcriterion}
\chi \in \span_H\left\{ \rho_\theta, \ 0 <\theta \le 1 \right\}.
\end{equation}
\end{theorem}

A proof of the if part of Theorem~\ref{th:NBcriterion} will be given at the beginning of Section~\ref{sec:dilated}.

\begin{remark} \rm 
Theorem~\ref{th:NBcriterion} is stated in a slightly different form than in the original papers \cite{Nym50}, \cite{Beu55}, in which the Hilbert space considered by the authors is $L^2(0,1)$. See \cite{BDBLS00} for the extension to the case of $H$. 
\end{remark}

Hence, RH holds if, given $\ve>0$, there exist $n \ge 1$, coefficients $c_1,\ldots,c_n \in \RR$, and $\theta_1,\ldots,\theta_n \in (0,1]$ such that 
\begin{equation} \label{eq:NBexplicit}
\int_0^\infty \left( \chi(t) - \sum_{k=1}^n c_k \left\{ \frac{\theta_k}{t} \right\}\right)^2 dt < \ve.
\end{equation}

Equation~\eqref{eq:NBexplicit} is reminiscent to the following convergence result:
\begin{equation} \label{eq:MobiusConvergence}
\lim_{n \rightarrow \infty} \sum_{k=1}^n \mu(k) \left\{\frac{1}{kt}\right\} = - \chi(t), \quad  t>0.
\end{equation}
This convergence holds point-wise and does not hold in $H$, see \cite[p.5-6]{BD99} for details, but this identity lead B\'aez-Duarte towards a stronger form of the Nyman-Beurling criterion, namely,

\begin{theorem}[\cite{BD03}] \label{th:BDcriterion}
RH holds if and only if
\begin{equation}\label{eq:BDcriterion}
\chi \in \span_H\left\{ \rho_{1/n},  n \ge1\right\}.
\end{equation}
\end{theorem}


The B\'aez-Duarte criterion (BD) can be restated as follows. For $n \ge  1$, let $\chi_n$ be the orthogonal projection of $\chi$ onto the linear subspace $H_n \subset H$ spanned by the family $(\rho_k)_{1 \le  k \le  n}$. The quantity
\begin{equation*}
d_n = \| \chi- \chi_n\|_H
\end{equation*}
is the distance between $\chi$ and $H_n$. Then RH holds if and only if $\lim_{n \rightarrow \infty} d_n = 0$ (It is furthermore equivalent to a particular asymptotic behaviour of the coefficients of $\chi_n$, see \cite{W07}). A stronger statement is actually conjectured, namely
\begin{equation*}
d_n^2 \sim \frac{C}{\log(n)},
\end{equation*}
where $C=2+\gamma-\log(4\pi)$, see~\cite{BS04}. Burnol proved the inequality $d_n^2 \ge \frac{C+o(1)}{\log(n)}$ for the same constant $C$, see \cite{Bur02}.
The inequality $d_n<\e$ provides zero-free regions for $\zeta$, see \cite{Nik95} for details when considering NB, and \cite{DFMR13} for more general results on Dirichlet series.

\medskip

A more general criterion has been stated in 2005 by B\'aez-Duarte, see \cite{BD05}. It is based on the M\"untz transform 
\begin{equation}
Pf(t)  =  \sum_{k\ge   1}f(tk) - \frac{1}{t} \int_0^{+\infty} f(x)dx, \quad f\in L^1(0,\infty),
\end{equation}
which is related to $\zeta$ via the M\"untz formula, see \cite[Theorem 3.1]{Bur07}, and \cite{Bur07} for a general study. B\'aez-Duarte considers in \cite{BD05} good kernels $f$, see \cite[Definition 1.1]{BD05} (this notion will be studied in Section \ref{gNBdilated}). The general strong B\'aez--Duarte criterion (gBD) is stated as follows:

\begin{theorem}\cite[Theorem 1.2]{BD05} \label{gBD}
Let $f:(0,\infty)\to \R$ be a good kernel. If RH holds then
\begin{equation}\label{eq:gBD}
f \in \span_H\left\{ t\mapsto Pf(nt),  n \ge1\right\}.
\end{equation}
Conversely, if \eqref{eq:gBD} is satisfied for a good kernel $f$ that is \emph{compactly supported} and whose Mellin transform has no zeros in the critical strip $\{1/2<\sg<1\}$, 
then RH holds.
\end{theorem}

The extension of the sufficient criterion for kernels $f$ that are not compactly supported is left open in~\cite{BD05}. The compactness assumption prevents from using analytical kernels, which are crucial for regularization and explicit calculations of the M\"untz transform. One contribution of our paper is the extension of Theorem~\ref{gBD} to a class of non-compactly supported kernels.




\subsection{The probabilistic and general criteria}

The basic idea of our work is to randomize the variables $\theta_k$ in NB, i.e. by replacing them by random variables. More precisely: 

\begin{definition} \label{def:pNBgNB} Given a family $(Z_{k,n})_{n \ge 1, 1 \le k \le n}$ of random variables and a family of coefficients $(c_{k,n})_{n \ge 1, 1 \le k \le n}$, we consider the distances
\bean
\cal D_n^2 & = & \EE \int_0^\infty\left(\chi(t)-\sum_{k=1}^n c_{k,n} \left\{\frac{Z_{k,n}}{t}\right\}\right)^2dt \qquad ({\rm pNB}),\\
D_n^2 & = & \int_0^\infty\left(\chi(t)-\sum_{k=1}^n c_{k,n} \EE \left\{\frac{Z_{k,n}}{t}\right\}\right)^2dt \qquad ({\rm gNB}).
\eean
We say that the family $(Z_{k,n})$ satisfies the probabilistic (resp. general) Nyman--Beurling criterion ${\rm pNB}(Z_{k,n})$ (resp. ${\rm gNB}(Z_{k,n})$), if one can find coefficients $(c_{k,n})$ such that $\cal D_n \to 0$ (resp. $D_n \to 0$).
\end{definition}

Let us notice that $D_n^2 \le \cal D_n^2$, due to $(\EE X)^2\le \EE X^2$,
which means that pNB implies gNB.\\


\subsection{Main results and Outline} 

In Section~\ref{sec:Erdos}, we show the implication pNB $\Longrightarrow$ RH under an assumption ($\cal P$) that is suited for random variables on $(0,1)$, but mild enough to be satisfied for some families supported on $(0,\infty)$. The proof of this implication is based on Erd\"os' probabilistic method.
In Section~\ref{sec:concentrated}, we study families of random variables that are more and more concentrated around $1/k$, $k\ge1$. For such families, we prove the equivalence pNB $\Longleftrightarrow$ RH. The intuitive underlying idea is that the random perturbation around $1/k$ is sufficiently small to be able to use a quantitative version of the B\'aez-Duarte criterion.

In Section~\ref{sec:dilated}, we prove the implication gNB $\Longrightarrow$ RH under a moment assumption on the sequence $(Z_{k,n})$. These r.v. may have a non-compact support, but the price to pay is a condition (C) on the growth of the coefficient $c_{k,n}$. 
We finally prove the implication RH $\Longrightarrow$ gNB + (C) for dilated r.v. using \cite[Theorem 1.2]{BD05} and a probabilistic interpretation of the M\"untz operator. 

In Section~\ref{sec:conclusion}, we illustrate our various criteria with exponential r.v. $\cal E(k)$ (prototype for dilated r.v.) and Gamma distribution $\Gamma(k,n)$ (prototype for concentrated r.v.).

\section{The pNB criterion} \label{sec:Erdos}

\subsection{Probabilistic framework and preliminaries} \

The Hilbert space
$
\HH = L^2(\Omega,H) \simeq L^2(\Omega \times (0,\infty))$ is endowed with the scalar product
\begin{equation*}
\langle Z,Z' \rangle_\HH = \EE \langle Z(\omega,\cdot), Z'(\omega,\cdot) \rangle_H  = \EE \int_0^\infty Z(\omega,t) Z'(\omega,t) dt.
\end{equation*}
To any random variable $X : \Omega \rightarrow \RR$, we associate the random Beurling function
$
\rho_X(t) = \left\{ \frac{X}{t} \right\},
$
which belongs to $\cal H$ when $X\in L^{1}_+(\Omega)$, see Lemma~\ref{prop:rhoL2} below.
We also introduce the "indicator random variable" $\chi(\omega,t) = \1_{[0,1]}(t)$, which is constant as an element of $L^2(\Omega,H)$.

A natural generalization of the deterministic Nyman-Beurling criterion to a probabilistic framework is the pNB criterion, defined in Definition \ref{def:pNBgNB}. 

As in the deterministic case, see \cite{BDBLS05}, the interesting point of such a criterion relies on the formula expressing the squared distance in $\cal H$ between $\chi$ and any subspace $\span_{\cal H}\{\rho_{Z_{k,n}}, 1\le k \le n\}$
as a quotient of Gram determinants.
One goal is then to figure out laws that reveal remarkable structures in the scalar products, leading to calculable determinants. \\

We recall that if $A\in\cal F$ is an event, then
$\pE [\1_A] =  \bP(A)$,
and that if $X\in L^1(\Omega)$, $X\ge0$, Fubini theorem yields the identity
\bean
\pE X =\pE\int_0^\infty \1_{t\le X} dt = \int_0^\infty \bP(X\ge t) dt.
\eean

\begin{prop} \label{prop:rhoL2}
Let $X\in L^1_+(\Omega)$. Then $\rho_X\in \cal H$ and 
\bean
\| \rho_X \|_{\cal H}^2 = \pE \int_0^\infty \left\{\frac{X}{t} \right\}^2 dt = (\log(2\pi)-\gamma)\ \|X\|_{L^1_+(\Omega)},
\eean
where $\gamma$ is the Euler constant.
\end{prop}

\begin{proof}
If $X(\omega)=0$ then $\d \int_0^\infty \left\{\frac{X(\omega)}{t} \right\}^2 dt=0$. By the change of variable $u=t/X(\omega)$ when $X(\omega)\neq 0$, we obtain
\bean
\pE \int_0^\infty \left\{\frac{X}{t} \right\}^2 dt = \pE \int_0^\infty \left\{\frac{1}{u} \right\}^2 Xdu 
    = \pE[X] \int_0^\infty \left\{\frac{1}{u} \right\}^2du.
\eean
The last integral can be bounded by $2$ but is actually computed in \cite[Prop.87 p.38]{BDBLS03}:  \bean
\int_0^\infty \left\{\frac{1}{u} \right\}^2du=\int_0^\infty \left\{t\right\}^2\frac{dt}{t^2}=\log(2\pi)-\gamma.
\eean
\end{proof}

\begin{prop} \label{prop:Malpha}
Let $Zù\in L^2_+(\Omega)$ and $\alpha\in(0,1)$. Then, for any $M>0$,
\bean
M^\alpha\int_{M}^\infty\pE\left\{\frac{Z}{t}\right\}^2dt & \le & \|Z\|_2^{1+\alpha} \int_0^\infty u^\alpha\left\{\frac{1}{u} \right\}^2du <\infty.
\eean
\end{prop}

\begin{proof}
By the change of variable $t=Zu$ and Fubini
\bean
\int_{M}^\infty\pE\left\{\frac{Z}{t}\right\}^2dt & = & \pE\int_{0}^\infty \1_{t\ge M}\left\{\frac{Z}{t}\right\}^2dt \\
    & = & \pE\int_{0}^\infty \1_{uZ\ge M}\left\{\frac{1}{u}\right\}^2 Zdu.
\eean
But, by the Cauchy-Schwarz and Markov's inequalities,
\bean
\pE[\1_{uZ\ge M}Z]\le \sqrt{\bP(uZ\ge M)}\|Z\|_2 \le \frac{u^\alpha}{M^\alpha}\sqrt{\pE Z^{2\alpha}}\|Z\|_2.
\eean
Noting that $\sqrt{\pE Z^{2\alpha}}=\|Z\|^\alpha_{2\alpha}\le \|Z\|_2^\alpha$ we obtained the desired result.
\end{proof}

\subsection{pNB implies RH under a mild condition} 

\begin{definition}
A family $(Z_{k,n})_{1 \le k \le n, n \ge 1}$ in $L^1_+(\Omega)$ is said to satisfy Assumption $(\cal P)$ if 
\bean
\exists \nu > 0 \ , \ \forall n \ge 1 \ , \ \PP(B_n) > \nu, \qquad (\cal P)
\eean
where 
\bean
B_n = \bigcap_{1 \le  k \le  n} \{0 < Z_{k,n} \le 1 \}.
\eean
\end{definition}
As an example, let $(X_k)_{k\ge 1}$ be a sequence of independent r.v. such that $X_k \sim \cal E(k)$. Then
\bean
\PP\left( B_n \right) = \prod_{k=1}^{n} (1-e^{-k}) 
\ge \prod_{k=1}^{\infty} (1-e^{-k})>0,
\eean
the later product being convergent since $\sum_{k\ge1}e^{-k}<\infty$. Thus, $(X_k)_{k\ge 1}$ verifies Assumption ($\cal P$).

\begin{theorem} \label{th:pNBtoRH}
Let $(Z_{k,n})_{1 \le k \le n, n \ge 1}$ be a collection of r.v. in $L^1_+(\Omega)$ satisfying Assumption $(\cal P)$ and the pNB criterion. Then RH holds.
\end{theorem}

The underlying idea of the proof consists in showing that the classical Nyman-Beurling criterion holds via Erd\"os' probabilistic method: in order to prove that an object exists, it suffices to show that it belongs to a set of positive measure, as explained in Chapter 1 of \cite{AS08}.

\begin{proof}
Fix $\ve>0$. Since $(Z_{k,n})_{1 \le k \le n , n \ge 1}$ satisfies pNB and Assumption $(\cal P)$, there exist $\nu >0$, $n\ge1$, and $c_{1,n},\cdots, c_{n,n}\in\R$ such that $\PP\left( B_n \right) >  \nu$ and
\bean
\cal D_n^2  =  \EE \left\|\chi(t)-\sum_{k=1}^n c_{k,n} \left\{\frac{Z_{k,n}}{t}\right\}\right\|_H^2 & < &\epsilon^2 \nu.
\eean
Let us consider the event
\bean
A_n =  \left\{\left\|\chi(t)-\sum_{k=1}^n c_{k,n} \left\{\frac{Z_{k,n}}{t}\right\}\right\|_H \le  \ve \right\}. 
\eean
By Markov's inequality, $\PP(^cA_n) \le \cal D_n^2/\ve^2$, and then 
\bean
\PP(A_n) \ge  1-\frac{\cal D_n^2}{\ve^2}>1-\nu.
\eean
Hence
$\PP(A_n) +\PP(B_n)  >   1$ and so
\begin{equation*}
\PP(A_{n} \cap B_{n}) = \PP(A_{n}) + \PP(B_{n}) - \PP(A_{n} \cup B_{n}) >0.
\end{equation*}
Thus, there exists $\omega\in\Omega$ such that, writing $\theta_{k}=Z_{k,n}(\omega)$,
\begin{equation*}
\left\|\chi(t)-\sum_{k=1}^n c_k \left\{\frac{\theta_k}{t}\right\}\right\|_H \le  \ve \ , \quad  0<\theta_{k}\le1, \quad 1\le  k\le  n.
\end{equation*}
Therefore, from the classical Nyman-Beurling criterion (Theorem~\ref{th:NBcriterion}), RH holds. \end{proof}

\subsubsection{A lower bound}

Let $(Z_{k,n})$ be a family of r.v. in $L^2_+(\Omega)$. Set
\bean
m_n  &=& \min_{1\le k\le n}Z_{k,n}, \\ M_n  &=&  \max_{1\le k\le n}Z_{k,n}, \\
\cal D_{n}  & = & \inf_{c_{1,n},\cdots,c_{n,n}}\left\|\chi(t)-\sum_{k=1}^n c_{k,n}\left\{\frac{Z_{k,n}}{t}\right\}\right\|_{\cal H}. 
\eean

\begin{lemma}\label{suffi}
Let $(Z_{k,n})$ satisfy Assumption $(\cal P)$. Then the following lower bound holds:
\bean
\cal D^2_{n} & \gg & \d \frac{1}{\log 2 + \pE \left|\log m_n \right|}.
\eean
\end{lemma}

\begin{proof}
Let us define 
\bean
B_\lb & = & \left\{ \sum_{k=1}^n c_{k,n} \rho_{\theta_k}, \ n\ge1, c_{k,n}\in\R, 0<\theta_k\le 1, \min_{1\le k\le n}\theta_k \ge\lambda \right\},
\eean
and $d(\lambda)$ the distance in $H$ between $\chi$ and $B_\lb $.
We recall a fundamental inequality obtained in \cite[p.131]{BDBLS00}:
\bean
d(\lb)^2 & \gg & \frac{1}{\log(2/\lb)}.
\eean
We then deduce that for all $n\ge1$, $c_{k,n} \in \R$ and almost surely,
\bean
\left\|\chi(t)-\sum_{k=1}^n c_{k,n}\left\{\frac{Z_{k,n}}{t}\right\}\right\|_H \1_{M_n\le1} & \gg & \frac{\1_{M_n\le1}}{\sqrt{\log(2/m_n)}}.
\eean
Since $(Z_{k,n})_{1\le k\le n, n\ge1}$ satisfies Assumption $(\cal P)$, $\pE \1_{M_n\le1} \gg 1$, and then by Cauchy-Schwarz inequality, 
\bean
1\ll \sqrt{\pE \left| \log\left(\frac{2}{m_n}\right) \right| \1_{M_n\le1}}  \sqrt{\pE \left\|\chi(t)-\sum_{k=1}^n c_{k,n}\left\{\frac{Z_{k,n}}{t}\right\}\right\|_H^2}.
\eean
In particular,
\bean
\cal D_{n} \sqrt{\pE \left|\log\left(\frac{2}{m_n}\right)\right|} & \gg & 1,
\eean
which yields the conclusion by the triangle inequality.
\end{proof}



\subsection{RH implies pNB under concentration} \label{sec:concentrated}

The goal of this section is to prove the following:

\begin{theorem} \label{th:RHpNB}
For any $n\ge 1$, let $(X_{k,n})_{1 \le  k \le  n}$ be r.v. in $L^1_+(\Omega)$ such that, setting $Y_{k,n}=\sqrt{X_{k,n}}$,
\bea
\EE \ Y_{k,n} & = &  \frac{1}{\sqrt{k}},\quad 1 \le  k \le  n, \label{expY}\\
\sup_{1\le k\le n} \var Y_{k,n} & \ll &  n^{-3-\vartheta},
\label{varY}\\
\bP(Y_{1,n} \ge 1) & \le & 1-\nu, \label{1erva}
\eea
for some $\vartheta>0$ and $\nu\in(0,1)$. Therefore, 
RH holds if and only if $(X_{k,n})_{n\ge1,k \in\llbracket 1,n \rrbracket}$ satisfies pNB.
\end{theorem}
One can check for instance that the r.v. $ Y_{k,n}\sim\Gamma\left(\frac{n^{3+\vartheta}}{k},\frac{n^{3+\vartheta}}{\sqrt{k}}\right)$, $\vartheta>0$, satisfy Conditions (\ref{expY}), (\ref{varY}) and \eqref{1erva}, see Section \ref{sec:conclusion} for discussions about examples.\\

In order to prove Theorem~\ref{th:RHpNB}, RH will be used at two different places and in two different ways, first via an explicit version of B\'aez-Duarte criterion; combining Proposition 1, 2 and 3 in \cite{BdR10}, one obtains

\begin{theorem}[\cite{BdR10}] \label{th:quantitativeBD}
For $\ve >0$ and $n \ge 1$, set
\begin{equation}
\nu_{n,\ve} = \left\|\chi(t)+\sum_{k \le n} \mu(k) k^{-\ve} \left\{\frac{1}{kt}\right\} \right\|_{H}^2.
\end{equation}
Under RH, the following limit holds
\begin{equation*}
\limsup_{n \rightarrow \infty} \nu_{n,\ve} \xrightarrow[\ve \rightarrow 0]{} 0.
\end{equation*}
\end{theorem}


In order to prove Theorem~\ref{th:RHpNB}, we will use some information about the coefficients in these linear combinations, namely that $|\mu(n)|\le 1$.
RH will be also used via the Lindel\"of hypothesis about the rate of growth of the $\zeta$ function on the critical line (see \cite[p. 336-337]{Tit86}):
\begin{theorem}[\cite{Tit86}] \label{th:Lindelof}
Under RH, the Lindel\"of hypothesis holds:
\begin{equation}
\left| \zeta\left(1/2+it\right)\right| \ll_\eta t^{\eta}, \quad \eta>0.
\end{equation}
\end{theorem}

We now turn to the proof of Theorem~\ref{th:RHpNB}.

\begin{proof}
To prove that pNB implies RH, by virtue of Theorem \ref{th:pNBtoRH}, it suffices to show that the family $(X_{k,n})_{1 \le  k \le  n}$ satisfies Assumption ($\cal P$). First, by union bound and Chebyshev's inequality, we have
\bean
\bP\left(\bigcup_{k=2}^n \{Y_{k,n} \ge 1 \}\right)
 \le \sum_{k=2}^n \bP(Y_{k,n} \ge 1) \le \sum_{k=2}^n \frac{\var Y_{k,n}}{(1-1/\sqrt{k})^2} \le
\frac{1}{(1-1/\sqrt{2})^2} \frac{n}{n^{3+\vartheta}} \le \frac{12}{n^{2+\vartheta}} \xrightarrow[n \rightarrow \infty]{} 0.
\eean
Moreover, we have $\bP(Y_{1,n} \ge 1)\le 1-\nu$ for all $n\ge1$. Therefore, for all $n$ sufficiently large, 
\bean
\bP\left(\bigcup_{k=1}^n \{Y_{k,n} \ge 1 \}\right) & \le & 1-\nu/2,
\eean
and then, taking the complement, ($\cal P$) holds.\\

Let us now prove that pNB($(X_{k,n})$) holds under RH. We have 
\begin{eqnarray*}
\inf_{a_1,\cdots,a_n}\EE\left\| \chi - \sum_{k=1}^n a_k \rho_{X_{k,n}}\right\|^2_H & \le & \EE\left\| \chi + \sum_{k=1}^n \mu(k) k^{-\ve} \rho_{X_{k,n}}\right\|^2_H \\
	& \le & \EE\left\| \chi + \sum_{k=1}^n \mu(k) k^{-\ve} \rho_{1/k} + \sum_{k=1}^n \mu(k) k^{-\ve} (\rho_{X_{k,n}}- \rho_{1/k})\right\|^2_H \\
	& \ll & \nu_{n,\ve} + \EE \left\| \sum_{k=1}^n\mu(k) k^{-\ve} (\rho_{X_{k,n}}-\rho_{1/k}) \right\|^2_H = \nu_{n,\ve}+ \cal R_{n,\ve}.
\end{eqnarray*}

It thus remains to study $\cal R_{n,\ve}$.
Using Plancherel's formula, see \cite[Prop.1]{BdR10}, we obtain
\begin{equation} \label{Rn}
\mathcal{R}_{n,\ve}  =  \EE\left\| \sum_{k=1}^n \mu(k) k^{-\ve} \left( \frac{1}{k^s}-X_{k,n}^s\right) \frac{\zeta(s)}{s} \right\|^2_{L^2}  =  \EE\int_{-T_n}^{T_n} V_n(t)dt+ \EE\int_{|t|\ge  T_n} V_n(t)dt,
\end{equation}
where
\begin{equation*}
V_n(t)  =  \left| \sum_{k=1}^n \mu(k) k^{-\ve} \left( \frac{1}{k^s}-X_{k,n}^s\right) \frac{\zeta(s)}{s}  \right|^2, \quad s=\frac12+it,
\end{equation*}
and where the parameter $T_n$ is to be chosen later. The quantity $V_n(t)$ depends on $\e$, but we do not mention this dependence as a subscript since we will bound it independently of $\e$ just below. 

\medskip

Let us recall the following useful inequality for $a,b\in\RR$, $\Re(s)=1/2$,
\bea \label{majexp}
|e^{as}-e^{bs}|=\left|\int_a^b s e^{us}du\right| \le |s| \left|\int_a^b e^{u/2}du \right| =2|s| \left| e^{a/2}-e^{b/2}\right|.
\eea
Using the Cauchy-Schwarz inequality, $|\mu(k) k^{-\ve}|\le 1$ and (\ref{majexp}), we obtain
\begin{eqnarray}
V_n(t) & \le & n \sum_{k=1}^n \left| \frac{1}{k^s}-X_{k,n}^s \right|^2  \frac{|\zeta(s)|^2}{|s|^2} \label{v1}\\
   & \le & 4 n \sum_{k=1}^n \left| \frac{1}{\sqrt{k}}-Y_{k,n} \right|^2  |\zeta(s)|^2.  \label{v2}
\end{eqnarray}

Let us consider the term $\EE\int_{-T_n}^{T_n}$ in (\ref{Rn}).
From (\ref{v2}), the Lindel\"of hypothesis (cf. Theorem~\ref{th:Lindelof}) written as $|\zeta(s)|\ll t^{\eta/2}$, and (\ref{varY}), we can write for any $\eta>0$,
\bea
\nonumber \EE\int_{-T_n}^{T_n} V_n(t)dt & \le & 4 n \sum_{k=1}^n \var(Y_{k,n})  \int_{-T_n}^{T_n}  |\zeta(s)|^2dt \\
\nonumber    & \ll & n^2\ \sup_{1\le k \le n}\var(Y_{k,n}) \ T_n^{1+\eta} \\
\label{t1}	& \ll & n^{-1-\vartheta}\ T_n^{1+\eta}.
\eea

We now study the term $\EE\int_{|t|\ge  T_n}$. From (\ref{v1}), we obtain
\begin{equation*}
\EE V_n(t)  \le  2 n \sum_{k=1}^n \left( \frac{1}{k}+\EE X_{k,n} \right)  \frac{|\zeta(s)|^2}{|s|^2}.
\end{equation*}
But $\EE X_{k,n}= \left(\EE Y_{k,n}\right)^2 + \var(Y_{k,n})\ll 1/k$, therefore, for any $\eta\in(0,1)$,
\bea
\label{t2} \EE\int_{|t|\ge  T_n} V_n(t)dt  &\ll & n \log n\  T_n^{\eta-1}.
\eea

We finally need to tune $\eta$ and $T_n$ accordingly. Recall that $\vartheta>0$ is given. Choose $\eta>0$ such that $ \frac{1+\eta}{1-\eta}<1+\vartheta/2$, and $\alpha>1$ so that 
\begin{equation*}
0<\frac{1}{1-\eta} < \alpha < \frac{1+\vartheta/2}{1+\eta}.
\end{equation*}
Set $T_n=n^\alpha$. Hence, from \eqref{t1} and \eqref{t2}
\bean
\EE\int_{|t|\ge  T_n} V_n(t)dt  & \ll & n^{1-\alpha(1-\eta)} \log n \xrightarrow[n\to \infty]{}0, \\
\EE\int_{-T_n}^{T_n} V_n(t)dt  & \ll & n^{-\vartheta/2} \xrightarrow[n\to \infty]{}0.
\eean
Finally, we conclude with Theorem \ref{th:quantitativeBD}.
\end{proof}

\section{The gNB criterion} \label{sec:dilated}

\subsection{Proof of the sufficient part in the deterministic criterion.}

It is first important to write a short proof of the sufficient implication in the NB criterion, stated in Theorem~\ref{th:NBcriterion} since the proof of Theorem~\ref{th:gNBtoRH} will follow a similar structure.
Let us prove that RH holds if
\begin{equation}\label{eq:NBcriterion}
\chi \in \span_H\left\{ \rho_\theta, \ 0 <\theta \le 1 \right\},
\end{equation}
by adapting the original proof, see e.g. \cite{Beu55}, with \cite{BDBLS00}. 
 
\begin{proof}
We recall the proof of this sufficient condition by Nyman and completed by an argument in \cite{BDBLS00}, see \cite[Lemme 1 \& Prop. 1 p.133]{BDBLS00}. Assume that (\ref{eq:NBcriterion}) is satisfied, i.e. that there exist coefficients $c_k=c_{k,n}$ and $\theta_k=\theta_{k,n}$ such that 
\begin{equation} \label{eq:NBexplicite}
d_n^2=\int_0^\infty \left( \chi(t) - \sum_{k=1}^n c_k \left\{ \frac{\theta_k}{t} \right\}\right)^2 dt  \xrightarrow[n\to\infty]{}  0.
\end{equation}
Let $s \in \mathbb{C}$ be such that $1/2<\sigma<1$ and assume for contradiction that $\zeta(s)=0$. We have:
\begin{equation}\label{eq:IdentiteNyman}
\int_0^\infty \left( \chi(t) - \sum_{k=1}^n c_k \left\{ \frac{\theta_k}{t} \right\}\right)t^{s-1}dt = \frac{1}{s}+\frac{\zeta(s)}{s}\sum_{k=1}^n c_k \theta_k^s = \frac{1}{s}.
\end{equation}
By Cauchy-Schwarz inequality, 
\bean
\left|\int_0^1 \left( \chi(t) - \sum_{k=1}^n c_k \left\{ \frac{\theta_k}{t} \right\}\right)t^{s-1}dt\right|^2 & \le & 
\int_0^1 \left( \chi(t) - \sum_{k=1}^n c_k \left\{ \frac{\theta_k}{t} \right\}\right)^2 dt \int_0^1 t^{2\sigma-2}dt \\
     & \le & \frac{d_n^2}{2\sigma-1}.
\eean
Moreover 
\bean
\left|\int_1^\infty \left( \chi(t) - \sum_{k=1}^n c_k \left\{ \frac{\theta_k}{t} \right\}\right)t^{s-1}dt\right|^2  = 
\left|\int_1^\infty  \sum_{k=1}^n c_k \theta_k t^{s-2}dt\right|^2  \le  \frac{1}{(1-\sigma)^2}\left|\sum_{k=1}^n c_k \theta_k \right|^2.
\eean
But 
\bean
\left|\sum_{k=1}^n c_k \theta_k \right|^2 =  \int_1^\infty \left( \chi(t) - \sum_{k=1}^n c_k \left\{ \frac{\theta_k}{t} \right\}\right)^2 dt  \le  d_n^2.
\eean
Hence, 
\bean
\int_0^\infty \left( \chi(t) - \sum_{k=1}^n c_k \left\{ \frac{\theta_k}{t} \right\}\right)t^{s-1}dt & \xrightarrow[n\to\infty]{} & 0,
\eean
which contradicts Eq. (\ref{eq:IdentiteNyman}).
\end{proof}

\subsection{gNB implies RH under a controlled approximation}

We can replace $\chi$ in Definition \ref{def:pNBgNB} by a more general function $\phi$. We say that $\phi : (0,\infty) \rightarrow \R$ is an admissible target function if 
\begin{enumerate}
    \item[(T1)] $\d \widehat{\phi}(s)=\int_0^\infty\phi(t)t^{s-1}dt$ exists and does not vanish in the strip $\frac{1}{2} < \sigma < 1$,
    \item[(T2)] $\d \sup_{M>0} \left(M \int_{M}^\infty \phi(t)^2 dt\right) < \infty$.
\end{enumerate}

\begin{theorem} \label{th:gNBtoRH}
Let $(Z_{k,n})_{1\le k\le n, n \ge 1}$ be r.v. in $L^2_+(\Omega)$ satisfying, for any $\epsilon >0$,
\bea \label{moment}
\sum_{k=1}^n \|Z_{k,n}\|_2^{1+\epsilon} \ll_\epsilon 1.
\eea
Let $\phi : (0,\infty) \rightarrow \R$ be an admissible target function.
We suppose that there exist coefficients $(c_{k,n})_{1 \le k \le n, n \ge 1}$ such that 
\begin{enumerate}
    \item[(gNB)]
$\d D_n^2 = \int_0^\infty \left|\phi(t)-\sum_{k=1}^n c_{k,n}\pE\left\{\frac{Z_{k,n}}{t}\right\} \right|^2 dt \xrightarrow[n \rightarrow \infty]{} 0$ ;
\item[(C)] For any $M_n\to\infty$, 
\bea \label{eq:cknBound}
\sum_{k=1}^n |c_{k,n}|^2\PP(Z_{k,n} \ge M_n) \xrightarrow[n\to\infty]{} 0.
\eea
\end{enumerate}
Then RH holds.
\end{theorem}

\begin{proof} We first compute the following Mellin transform:
\bean 
\int_0^\infty \left(\phi(t)- \sum_{k=1}^n c_{k,n} \pE\left\{\frac{Z_{k,n}}{t}\right\}\right) t^{s-1} dt 
& = & \widehat{\phi}(s) +  \frac{\zeta(s)}{s} \sum_{k=1}^n c_{k,n} \pE Z_{k,n}^s.
\eean
Suppose for contradiction that $\zeta(s)=0$ for some fixed $s$ with $\frac{1}{2}<\sigma<1$. We thus have from (T1),
\bea \label{eq:MellingNB}
0<\left|\widehat{\phi}(s) \right|^2 = \left| \int_0^\infty \left(\phi(t)- \sum_{k=1}^n c_{k,n} \pE\left\{\frac{Z_{k,n}}{t}\right\}\right) t^{s-1} dt \right|^2 = |I_n|^2.
\eea 
 We will prove that the right-hand side of~\eqref{eq:MellingNB} goes to $0$ as $n \rightarrow \infty$, which contradicts $\widehat{\phi}(s)\neq 0$.
 
\medskip

We split $I_n$ and use the inequality
\bean
|I_n|^2 &\ll& |I_{1,n}|^2+|I_{2,n}|^2 \\&=& \left|\int_0^{M_n} \left(\phi(t)- \sum_{k=1}^n c_{k,n} \pE\left\{\frac{Z_{k,n}}{t}\right\}\right) t^{s-1} dt\right|^2+\left|\int_{M_n}^\infty \left(\phi(t)-\sum_{k=1}^n c_{k,n} \pE\left\{\frac{Z_{k,n}}{t}\right\}\right) t^{s-1} dt\right|^2,
\eean
where the moving threshold $M_n \ge 1$ is chosen so that
\begin{itemize}
    \item $M_n \xrightarrow[n \rightarrow \infty]{} \infty$,
    \item $M_n^{2 \sigma-1} D_n^2 \xrightarrow[n \rightarrow \infty]{} 0$ (this is possible since we assume $D_n\to 0$).
\end{itemize}
The first integral is bounded with the Cauchy-Schwarz inequality:
\bean
\left|I_{1,n} \right|^2
& \le &  D_n^2 \int_0^{M_n} t^{2\sg -2}dt =  D_{n}^2 \frac{M_n^{2\sg-1}}{2\sg-1}\xrightarrow[n \rightarrow \infty]{} 0.
\eean
In order to bound $|I_{2,n}|$, we write
\bean
I_{2,n} &=& \int_{M_n}^\infty \phi(t) t^{s-1} dt - \sum_{k=1}^n c_{k,n} \pE \int_{M_n}^\infty \left\{\frac{Z_{k,n}}{t}\right\} t^{s-1} dt.
\eean
As $\widehat{\phi}(s)$ is well-defined, we have $\int_{M_n}^\infty \phi(t) t^{s-1} dt \xrightarrow[n \rightarrow \infty]{} 0$. We split the other integrals:
\bean
\pE \int_{M_n}^\infty \left\{\frac{Z_{k,n}}{t}\right\} t^{s-1} dt & = &
\pE \int_{M_n}^\infty \1_{Z_{k,n}\le M_n}\left\{\frac{Z_{k,n}}{t}\right\} t^{s-1} dt + \pE \int_{M_n}^\infty \1_{Z_{k,n}\ge M_n}\left\{\frac{Z_{k,n}}{t}\right\} t^{s-1} dt \\
 & = & A_{k,n} + B_{k,n}.
\eean
(1) We first bound the term $\d \sum_{k=1}^n c_{k,n} B_{k,n}$, by splitting again each integral:
\bea
B_{k,n} & = & \pE \int_{M_n}^{Z_{k,n}} \1_{Z_{k,n} \ge M_n}\left\{\frac{Z_{k,n}}{t}\right\} t^{s-1} dt 
+ \pE \int_{Z_{k,n}}^\infty \1_{Z_{k,n} \ge M_n} \left\{\frac{Z_{k,n}}{t}\right\} t^{s-1} dt.
\eea
For the first integral, we use the bound $\left|\left\{\frac{Z_{k,n}}{t}\right\}\right| \le 1$. For the second integral, we notice that $\left\{\frac{Z_{k,n}}{t}\right\} = \frac{Z_{k,n}}{t}$ when $t \ge Z_{k,n}$. The triangle inequality then gives:
\bea
|B_{k,n}| & \le & \pE  \1_{Z_{k,n} \ge M_n} \int_{M_n}^{Z_{k,n}} t^{\sigma-1} dt + \pE  \1_{Z_{k,n} \ge M_n} Z_{k,n} \int_{Z_{k,n}}^\infty t^{\sigma-2} dt \\
 & \ll & \pE  \1_{Z_{k,n} \ge M_n} Z_{k,n}^{\sigma}.
 \eea
 Thus, by the Cauchy-Schwarz inequality,
 \bea \label{changegamma}
 |B_{k,n}| &\le & \sqrt{\PP(Z_{k,n} \ge M_n)}\ \sqrt{\pE Z_{k,n}^{2 \sigma}}.
\eea
Thus, by the triangle and Cauchy-Schwarz inequalities again,
\bean
\left|\sum_{k=1}^n c_{k,n}B_{k,n}\right|^2 & \le & 
\sum_{k=1}^n |c_{k,n}|^2 \PP(Z_{k,n} \ge M_n) \sum_{k=1}^n \pE Z_{k,n}^{2 \sigma}.
\eean
Let us notice that $\pE Z_{k,n}^{2 \sigma}=\|Z_{k,n}\|_{2\sg}^{2\sg}\le \|Z_{k,n}\|_{2}^{2\sg}$ since $2\sg\le 2$. Since $2\sg>1$, we can deduce from \eqref{moment} and \eqref{eq:cknBound} that 
\bean
\sum_{k=1}^n c_{k,n}B_{k,n} & \xrightarrow[n\to\infty]{} & 
0.
\eean
(2) It remains to bound the term $\d \sum_{k=1}^n c_{k,n} A_{k,n}$. We notice that, for $t\ge M_n$,
\bean \1_{Z_{k,n} \le M_n} \left\{\frac{Z_{k,n}}{t}\right\} = \1_{Z_{k,n} \le M_n} \frac{Z_{k,n}}{t},
\eean
so
\bea
\left|\sum_{k=1}^n c_{k,n}A_{k,n}\right|^2 &=& \left|\sum_{k=1}^n c_{k,n} \pE \1_{Z_{k,n} \le M_n}Z_{k,n} \int_{M_n}^\infty t^{s-2} dt \right|^2\\ & \ll & \left|\sum_{k=1}^n c_{k,n} \pE \1_{Z_{k,n} \le M_n} Z_{k,n}\right|^2 M_n^{2\sigma-2}.
\eea
But
\bean
M_n^{2\sigma-2} \left|\sum_{k=1}^n c_{k,n} \pE  \1_{Z_{k,n}\le M_n}Z_{k,n} \right|^2  & = & M_n^{2\sigma-1} \int_{M_n}^\infty \left(\sum_{k=1}^n c_{k,n} \pE \1_{Z_{k,n}\le M_n} \frac{Z_{k,n}}{t}\right)^2 dt \\
& = & M_n^{2 \sigma -1} \int_{M_n}^\infty \left(\sum_{k=1}^n c_{k,n} \pE \1_{Z_{k,n} \le M_n} \left\{\frac{Z_{k,n}}{t}\right\}\right)^2 dt,
\eean
and
\bean
\int_{M_n}^\infty \left(\sum_{k=1}^n c_{k,n} \pE \1_{Z_{k,n} \le M_n} \left\{\frac{Z_{k,n}}{t}\right\}\right)^2 dt   \ll D_n^2 + \int_{M_n}^\infty \phi(t)^2 dt  + \int_{M_n}^\infty \left(\sum_{k=1}^n c_{k,n} \pE \1_{Z_{k,n} \ge M_n} \left\{\frac{Z_{k,n}}{t}\right\}\right)^2 dt.
\eean
The sequence $(M_n)_{n \ge 1}$ has been chosen such that $M_n^{2 \sigma-1}D_n^2 \xrightarrow[n \rightarrow \infty]{} 0$. Due to assumption (T2) we have $\d M_n^{2\sigma-1} \int_{M_n}^\infty \phi(t)^2 dt \xrightarrow[n \rightarrow \infty]{} 0$ since $2\sg-1<1$. Let us bound the third term:
\bean
\pE  \1_{Z_{k,n} \ge M_n} \left\{\frac{Z_{k,n}}{t}\right\} & \le & \sqrt{\PP(Z_{k,n} \ge M_n)}\ \sqrt{\pE \left\{\frac{Z_{k,n}}{t}\right\}^2}.
\eean
Therefore
\bean
\left(\sum_{k=1}^n c_{k,n} \pE \1_{Z_{k,n}\ge M_n} \left\{\frac{Z_{k,n}}{t}\right\}\right)^2  dt 
& \le & \sum_{k=1}^n |c_{k,n}|^2\PP(Z_{k,n} \ge M_n)  \sum_{k=1}^n \pE\left\{\frac{Z_{k,n}}{t}\right\}^2dt,
\eean
and
\bean
\int_{M_n}^\infty \left(\sum_{k=1}^n c_{k,n} \pE \1_{Z_{k,n}\ge M_n} \left\{\frac{Z_{k,n}}{t}\right\}\right)^2  dt      & \le & \sum_{k=1}^n |c_{k,n}|^2\PP(Z_{k,n} \ge M_n) \sum_{k=1}^n \int_{M_n}^\infty\pE\left\{\frac{Z_{k,n}}{t}\right\}^2dt.
\eean
Due to Proposition~\ref{prop:Malpha} (take $\alpha=2\sigma -1$), we have
\bean
M_n^{2\sigma -1}\pE\int_{M_n}^\infty\left\{\frac{Z_{k,n}}{t}\right\}^2 dt & \ll_\sg & \|Z_{k,n}\|_2^{2\sg}.
\eean
Since $2\sg>1$ we can use \eqref{moment}, together  with condition (\ref{eq:cknBound}), to obtain
\bean 
M_n^{2\sigma -1} \int_{M_n}^\infty \left(\sum_{k=1}^n c_{k,n} \pE \1_{Z_{k,n}\ge M_n} \left\{\frac{Z_{k,n}}{t}\right\}\right)^2 dt \xrightarrow[n \rightarrow \infty]{} 0.
\eean 
(3) Hence $|I_{2,n}|\xrightarrow[n \rightarrow \infty]{} 0$, which concludes the proof.
\end{proof}

\subsection{RH implies gNB for dilated random variables}\label{gNBdilated}

The proof of the implication RH $\Longrightarrow$ gNB for dilated r.v. $Z_k=X/k$ is based on the Necessary part of B\'aez-Duarte's theorem \cite[Theorem 1.2]{BD05} regarding M\"untz transform $P$.
We first give a probabilistic interpretation of $P$.

\begin{lemma}\label{prop:EP}
Let $X\in L^1_+(\Omega)$ and set $f(x)=\PP(X\ge  x)$, $x\ge0$. Then $Pf$ is well defined and 
\begin{equation} \label{eq:EPf}
\EE[\{X/t\}]  =  -Pf(t), \quad t>0.
\end{equation}
\end{lemma}

\begin{proof}
First, notice that $0\le f(k+1)\le \int_0^1 f(k+x)dx$, $k\ge 0$, so that the following quantities are well defined: 
\bean
\sum_{k\ge 0}f(k+1) \le \int_0^\infty f(x)dx =\EE X <\infty.
\eean
Since $0\le  \{X\}\le  1$, we can write
\begin{eqnarray*}
\EE [\{X\}] & = & \int_0^1 \PP(\{X\}\ge  x)dx \\     & = & \int_0^1 \sum_{k\ge   0} \PP(k+x\le  X<k+1)dx  =  \int_0^1 \sum_{k\ge   0} (f(k+x)-f(k+1))dx \\
    & = & \sum_{k\ge   0} \int_0^1 f(k+x)dx - \sum_{k\ge   0}f(k+1) = \int_0^{+\infty} f(x)dx - \sum_{k\ge   1}f(k).
\end{eqnarray*}
Set $t>0$. Then $\PP(X/t\ge   x)=f(tx)$ and so 
\begin{equation}
\EE[\{X/t\}] = \int_0^{+\infty} f(tx)dx - \sum_{k\ge   1}f(tk) = \frac{1}{t} \int_0^{+\infty} f(x)dx - \sum_{k\ge   1}f(tk),
\end{equation}
as desired.
\end{proof}

B\'aez-Duarte introduced in \cite[Definition 1.1]{BD05} the definition of a good kernel $f$, i.e. $f$ is a continuously differentiable function on $(0,\infty)$ with $\int_0^\infty|f(t)|dt<\infty$ and  $\int_0^\infty t|f'(t)|dt<\infty$.
Let us notice that if $X$ is a positive integrable r.v. with a continuous density $\phi$ then $f(t)=\bP(X\ge t)$ is a good kernel, and $f'=-\phi$.
We also obtain the probabilistic interpretation of the formula (2.4) in \cite{BD05}:
\begin{equation*}
Pf(t)=\sum_{k\ge 1}f(kt)-\int_0^\infty f(ut)du = t\int_0^\infty f'(ut)\{u\}du=\int_0^\infty \{x/t\} f'(x)dx = -\EE[\{X/t\}].
\end{equation*}

\begin{theorem}\label{Dn}
Let $X\in L^{q}_+(\Omega)$, $q>1$, be a r.v. with a continuous density.\\ 
If RH holds,
then there exist coefficients $c_{k,n}$ such that
\begin{enumerate}
    \item[(gNB)]
$\d D_n^2 = \int_0^\infty \left|\bP(X\ge t)-\sum_{k=1}^n c_{k,n}\pE\left\{\frac{X}{kt}\right\} \right|^2 dt \xrightarrow[n \rightarrow \infty]{} 0$ ;
\item[(C)] For any $M_n\to\infty$, 
\bean 
\sum_{k=1}^n |c_{k,n}|^2\PP(X/k \ge M_n) \xrightarrow[n\to\infty]{} 0.
\eean
\end{enumerate}
\end{theorem}

\begin{proof}
Since RH holds, there exists coefficients $c_{k,n}$ bounded in $k$ and $n$ (see \cite{BdR10}), such that 
 \bean
 \int_0^\infty\left(\chi(t)-\sum_{k=1}^n c_{k,n}\{1/kt\}\right)^2dt \xrightarrow[n\to\infty]{} 0.
 \eean
 Then, B\'aez-Duarte deduces in Section 3.1 in \cite{BD05} that, for a good kernel $f$,
 \bean
 \int_0^\infty\left(f(t)-\sum_{k=1}^n c_{k,n}Pf(kt)\right)^2dt \xrightarrow[n\to\infty]{} 0.
 \eean
 Hence using our Lemma \ref{prop:EP}, we deduce that (gNB) holds for the r.v. $X/k$ and the target function $f:t\mapsto \bP(X\ge t)$, which is a good kernel (see above). 
 
 Condition (C) then follows from the boundedness of the coefficients $c_{k,n}$ and the inequality 
 \bean
 \PP(X/k \ge M_n) & \le & \frac{\pE X^{q}}{k^{q}M_n^{q}},
 \eean
 since $q>1$.
\end{proof}

\section{Examples} \label{sec:conclusion}

To illustrate our theorems, we give two typical examples:
\begin{enumerate}
    \item[(1)]\emph{Dilation}: Let $X_k=X/k$ where $X\sim\cal E(1)$. \\
    We have $\pE X^{q}<\infty$, $q>1$, and $\d \|X_{k}\|_2=\frac{\sqrt{2}}{k}$, so we can apply Theorem \ref{Dn} and Theorem \ref{th:gNBtoRH}.
    \item[(2)]\emph{Concentration}: Let $Z_{k,n}=Y_{k,n}^2$ where $Y_{k,n}\sim \Gamma\left(\frac{n^4}{k},\frac{n^4}{\sqrt{k}}\right)$, $1\le k\le n$. \\
    We have $\pE Y_{k,n} =1/\sqrt{k}$ and $\vari( Y_{k,n})=n^{-4}$. Since $Y_{1,n}$ is distributed as $\frac{E_1+\cdots+E_{n^4}}{n^4}$ where the $E_k$'s are i.i.d. $\mathcal{E}(1)$, the Central Limit Theorem gives $\bP(Y_{1,n}\ge 1) \rightarrow 1/2 < 1$. So we can apply Theorem \ref{th:RHpNB}. Moreover $\|Z_{k,n}\|_2=\|Y_{k,n}\|_4^2 = \pE[Y_{k,n}^4]^{1/2} \ll \frac{1}{k}+\frac{1}{n^2}$ ($Y_{k,n}$ is concentrated around $1/\sqrt{k}$ as $n$ growing), so Assumption \eqref{moment} in Theorem \ref{th:gNBtoRH} is verified. \\
\end{enumerate}

We summarize below the relationships between the various criteria :

\begin{center}\label{implication}
$\begin{array}{ccccc}
  &  &  &  & {\rm gNB}\left(\Gamma\left(\frac{n^4}{k},\frac{n^4}{\sqrt{k}}\right)^2_{1\le k\le n}\right) + (C) \\
  & & & \rotatebox[origin=c]{225}{\(\Longrightarrow\)}  & \Uparrow\\
{\rm gNB}\left(\cal E(k)_{k\ge 1}\right) +(C) & \Longleftrightarrow  &  {\rm RH} & \Longrightarrow & {\rm pNB}\left(\Gamma\left(\frac{n^4}{k},\frac{n^4}{\sqrt{k}}\right)^2_{1\le k\le n}\right) +(C) \\
  &   &  \Updownarrow & &  \\
  &  &  {\rm NB} &  & 
\end{array}$
\captionof{figure}{Links between the various criteria.}
\end{center}

Notice that condition (C) is not necessary for the implication $${\rm pNB}\left(\Gamma\left(\frac{n^4}{k},\frac{n^4}{\sqrt{k}}\right)^2_{1\le k\le n}\right)\Longrightarrow RH,$$ which is one of the interests of pNB. The correlation structure of these r.v., which is not explored here, might also be of some importance in pNB.

The computation of the scalar products in Example (1) is studied in \cite{DH20}: the main formula shows a striking simplification compared to Vasyunin's formula \cite{Vas95}.


\section*{Acknowledgement} The authors are very grateful to the anonymous referee for his careful reading, understanding and valuable remarks that improved this paper. The first author
warmly thanks Michel Balazard and \'Eric Saias for numerous engaging conversations over many years, especially preceding the first version of the paper.


\begin{thebibliography}{5}

\bibitem[AS08]{AS08} N. Alon, J.H. Spencer. The probabilistic method.
Third edition.
With an appendix on the life and work of Paul Erd\"os.
Wiley-Interscience Series in Discrete Mathematics and Optimization. John Wiley \& Sons, Inc., Hoboken, NJ, 2008.



\bibitem[BD99]{BD99}L. B\'aez-Duarte. A class of invariant unitary operators. {\em Adv. Math.}, 144 (1999), no. 1, 1--12.

\bibitem[BD03]{BD03} L. B{\'a}ez-Duarte. A strengthening of the Nyman-Beurling criterion for the Riemann hypothesis, {\em Rend. Mat. Ac. Lincei}, S. 9, 14 (2003) 1, 5-11.

\bibitem[BD05]{BD05}L. B{\'a}ez-Duarte. A general strong Nyman-Beurling criterion for the Riemann hypothesis. {\em Publications de l'Institut Math\'ematique}, Nouvelle S\'erie, 78 (2005), pp. 117--125.

\bibitem[BDBLS00]{BDBLS00} L. B{\'a}ez-Duarte, M. Balazard, B. Landreau and \'E. Saias. Notes sur la fonction $\zeta$ de Riemann, 3.
(French) [Notes on the Riemann $\zeta$-function, 3] {\em Adv. Math.}, 149 (2000), no. 1, 130--144.

\bibitem[BDBLS05]{BDBLS05} L. B{\'a}ez-Duarte, M. Balazard, B. Landreau and \'E. Saias. \'{E}tude de l'autocorr{\'e}lation multiplicative de la fonction "partie fractionnaire". (French) {\em The Ramanujan Journal}, 9(1) (2005), pp. 215--240.

\bibitem[BDBLS03]{BDBLS03} L. B{\'a}ez-Duarte, M. Balazard, B. Landreau and \'E. Saias. Document de travail -- \'{E}tude de l'autocorr{\'e}lation multiplicative de la fonction "partie fractionnaire". (French)

\bibitem[Bal10]{Bal10} M. Balazard. Un si\`ecle et demi de recherches sur l'hypoth\`ese de Riemann. {\em La Gazette des math\'ematiques}, 126 (2010), pp.7--24.


\bibitem[BdR10]{BdR10} M. Balazard and A. de Roton. Sur un crit\`ere de B{\'a}ez-Duarte pour l'hypoth\`ese de Riemann. {\em International Journal of Number Theory}, 6(04) (2010), pp. 883--903.

\bibitem[BS04]{BS04} M. Balazard, and \'E. Saias. Notes sur la fonction $\zeta$ de Riemann, 4. {\em Advances in Mathematics}, 188(1) (2004), pp. 69--86.

		

\bibitem[Beu55]{Beu55} A. Beurling. A closure problem related to the Riemann Zeta-function. {\em Proceedings of the National Academy of Sciences}, 41(5) (1955), pp. 312--314.

\bibitem[Bur02]{Bur02} J.F. Burnol. A lower bound in an approximation problem involving the zeros of the Riemann zeta function. {\em Advances in Mathematics}, 170(1) (2002), pp.56--70.

\bibitem[Bur07]{Bur07} J.F. Burnol. Entrelacement de co-Poisson.
(French)  [Co-Poisson links] {\em Ann. Inst. Fourier (Grenoble)} 57 (2007), no. 2, 525--602.

\bibitem[Con03]{Con03} J.B. Conrey. The Riemann hypothesis.
{\em Notices Amer. Math. Soc.}, 50 (2003), no. 3, 341--353.

\bibitem[DH20]{DH20} S. Darses and E. Hillion. An exponentially-averaged Vasyunin formula. {\em Proc. of the American Math. Soc.} To appear https://doi.org/10.1090/proc/15422.

\bibitem[DFMR13]{DFMR13} C. Delaunay, E. Fricain, E. Mosaki, O. Robert.  Zero-free regions for Dirichlet series. {\em Trans. Amer. Math. Soc.}, 365 (2013), no. 6, 3227--3253.

\bibitem[Nik95]{Nik95} N.Nikolski. Distance formulae and invariant subspaces, with an application to localization of zeros of the Riemann $\zeta $-function. {\em Annales de l'institut Fourier} Vol. 45, No. 1 (1995), pp. 143-159.

\bibitem[Nym50]{Nym50} B. Nyman. On the one-dimensional translation group and semi-group in certain function spaces. {\em Thesis, University of Uppsala}, 1950.

\bibitem[Ten95]{Ten95} G. Tenenbaum. Introduction \`a la th\'eorie analytique et probabiliste des nombres, Soci\'et\'e math\'ematique de France, 1995.

\bibitem[Tit86]{Tit86} E. C. Titchmarsh, The theory of the Riemann zeta-function, second ed., The Clarendon Press Oxford University Press, New York, 1986.

\bibitem[Vas95]{Vas95} V.I. Vasyunin. On a biorthogonal system associated with the Riemann hypothesis. (Russian) {\em Algebra i Analiz} 7, no. 3 (1995): 118-35; translation in {\em St. Petersburg Mathematical Journal} 7, no. 3 (1996): 405-19.

\bibitem[W07]{W07} A. Weingartner. On a question of Balazard and Saias related to the Riemann hypothesis. {\em Adv. Math.} 208 (2007), no. 2, 905--908.



\end{thebibliography}
\end{document}